%% file: Fregean_abstraction_deflationary_account.tex
\documentclass{amsart}

\title[Fregean abstraction in ZF: a deflationary account]{Fregean abstraction in Zermelo-Fraenkel set theory: a deflationary account}

\author{Joel David Hamkins}
\address[Joel David Hamkins]
{O'Hara Professor of Philosophy and Mathematics, University of Notre Dame, 100 Malloy Hall, Notre Dame, IN 46556 USA \&\ Associate Faculty Member, Professor of Logic, Faculty of Philosophy, University of Oxford, UK}
\email{jdhamkins@nd.edu}
\urladdr{http://jdh.hamkins.org}

\thanks{I am grateful to Paddy Blanchette for insightful discussion and for pointing me to several relevant passages in Frege. Commentary can be made about this article on my blog at \href{http://jdh.hamkins.org/fregean-abstraction-deflationary-account}{http://jdh.hamkins.org/fregean-abstraction-deflationary-account}.}

\usepackage{latexsym,amsfonts,amsmath,amssymb,mathrsfs}
\usepackage{bbm}
\usepackage[hidelinks]{hyperref}
\usepackage{relsize}
\usepackage[rgb,dvipsnames]{xcolor}
\usepackage{tikz}
\usepackage{pgflibrarysnakes}
\usetikzlibrary{arrows,arrows.meta,petri,topaths,positioning,shapes,shapes.misc,patterns,calc,decorations.pathreplacing,hobby,snakes}
\usepackage{wrapfig} 
\usepackage{float}
\usepackage[utf8]{inputenc}
\RequirePackage{doi}

\usepackage{enumitem}

\include{MathMacrosJDH}
\usepackage[backend=bibtex,style=alphabetic,maxbibnames=15,maxcitenames=6,dateabbrev=false]{biblatex}
\renewcommand{\UrlFont}{} 
\renewbibmacro{in:}{\ifentrytype{article}{}{\printtext{\bibstring{in}\intitlepunct}}} 
\DeclareFieldFormat{url}{{\UrlFont\url{#1}}} 
\DeclareFieldFormat{urldate}{
  (version \thefield{urlday}\addspace%
  \mkbibmonth{\thefield{urlmonth}}\addspace%
  \thefield{urlyear}\isdot)}
\DeclareFieldFormat{eprint:arxiv}{
\ifhyperref
    {\href{http://arxiv.org/abs/#1}{%
    arXiv\addcolon\nolinkurl{#1}}\iffieldundef{eprintclass}{}{\UrlFont{\mkbibbrackets{\thefield{eprintclass}}}}}
    {arXiv\addcolon\nolinkurl{#1}\iffieldundef{eprintclass}{}{\UrlFont{\mkbibbrackets{\thefield{eprintclass}}}}}}
\bibliography{HamkinsBiblio,MathBiblio,PhilBiblio,WebPosts}
\newcommand{\num}{\#}

\begin{document}

\begin{abstract}
The standard treatment of sets and definable classes in first-order Zermelo-Fraenkel set theory accords in many respects with the Fregean foundational framework, such as the distinction between objects and concepts. Nevertheless, in set theory we may define an explicit association of definable classes with set objects $F\mapsto\varepsilon F$ in such a way, I shall prove, to realize Frege's Basic Law~V as a \ZF\ theorem scheme, Russell notwithstanding. A similar analysis applies to the Cantor-Hume principle and to Fregean abstraction generally. Because these extension and abstraction operators are definable, they provide a deflationary account of Fregean abstraction, one expressible in and reducible to set theory---every assertion in the language of set theory allowing the extension and abstraction operators $\varepsilon F$, $\num G$, $\alpha H$ is equivalent to an assertion not using them. The analysis thus sidesteps Russell's argument, which is revealed not as a refutation of Basic Law~V as such, but rather as a version of Tarski's theorem on the nondefinability of truth, showing that the proto-truth-predicate ``$x$ falls under the concept of which $y$ is the extension'' is not expressible.
\end{abstract}

\maketitle

\section{Introduction}

The treatment of sets and classes in Zermelo-Fraenkel set theory instantiates in many respects the principal features of the Fregean foundational framework, in which Frege distinguishes sharply between \emph{objects} and \emph{concepts}. The realm of objects, a universe of individuals, constitutes the domain of discourse, a domain over which the quantifiers may range, and this domain is considered under the guise of diverse Fregean concepts, each serving in effect as a predicate on that domain, picking out a collection of those objects, for a given object either falls under a concept or does not. Frege analyzes concepts through their course of values, viewing them in effect as functions mapping every object either to an object standing for truth or one for falsity.

The usual practice and development of sets and classes in Zermelo-Fraenkel set theory accords very well with this Fregean framework, for set theorists consider the cumulative set-theoretic universe $V$, the collection of all sets, as constituting a domain of individual objects, the sets, about which we make mathematical assertions, with our quantifiers $\forall x$ and $\exists x$ ranging over all the sets in this universe, and we consider these sets in the light of the various set-theoretic properties $\varphi(x)$ expressible in the language of set theory, forming the corresponding definable classes $\set{x\mid\varphi(x)}$, thereby undertaking predication by those properties. Thus, we have a universe of first-order objects, the sets, in the context of diverse Fregean concepts, predicates on that domain, the classes. We refer to or name such a concept or class $\set{x\mid\varphi(x)}$ by providing the definition $\varphi$ along with any parameters, if any, used in that definition. Different intensional descriptions $\varphi$ or $\phi$ may end up describing the same class extensionally, if $\varphi(x)\iff \phi(x)$ holds for every individual $x$. In this way, we may come to view the usual understanding of sets and classes in \ZF\ set theory in a Fregean light.

Andreas Blass expresses this view, that proper classes are not part of the set-theoretic ontology, but rather are merely a convenience for predication.
\begin{quotation}
  Proper classes are not objects. They do not exist. Talking about them is a convenient abbreviation for certain statements about sets. (For example, V=L abbreviates ``all sets are constructible.'') If proper classes were objects, they should be included among the sets, and the cumulative hierarchy should%
  \ldots continue much farther, but in fact, it already continues arbitrarily far.
  \ldots%
  %
  Considerations like these are what prompt me to view ZFC as a reasonable foundational system, in contrast to Morse-Kelley set theory.~\cite{Blass.MO71773:Are-proper-classes-objects}

  %
\end{quotation}

Thus, Blass appears to work in precisely the set theory I have described, a set-theoretic realm having only sets as objects, and considering classes only as a shorthand manner of referring to definable predication.

In his well-known treatise on large cardinals, Kanamori~\cite{Kanamori2004:TheHigherInfinite2ed} similarly works in this theory. Notably, he regards the famous Kunen inconsistency, asserting that there is no nontrivial elementary embedding $j:V\to V$, as ``a schema of theorems, one for each $j$.'' That is, he considers the theorem as a scheme of statements making separate claims for each possible definition of such a class $j$.\footnote{I find it curious for Kanamori to emphasize this point with this particular theorem, however, in light of the fact that one may easily refute the existence of nontrivial \emph{definable} elementary embeddings $j:V\to V$ with a soft minimality argument showing that no particular definition (allowing parameters) can admit a smallest critical point of such an embedding; see~\cite{Suzuki1998:NojVtoVinV[G], Suzuki1999:NoDefinablejVtoVinZF}. This soft argument furthermore makes no use of the axiom of choice, while for the class version this is a major open question leading to the so-called choiceless large cardinals.
}

In this article, I aim to investigate set theory from this Fregean perspective, to see how well we can implement his ideas in Zermelo-Fraenkel set theory. Let me emphasize, however, that my project here is not to elucidate Frege's system in set theory, but rather to be inspired by Fregean ideas as a way of understanding what is going on in set theory. In short, I aim to investigate Fregean-inspired set theory rather than to present a set-theoretical Frege.

\section{Basic Law~V}

One of Frege's central ideas was that we might associate every concept $F$ with an object $\varepsilon F$, the \emph{extension} of $F$, an object that represents this concept. In this way, Frege is able to treat concepts as objects, by means of their extensions---he can quantifier over them and refer to them by means of the objects that represent them.

The \emph{extension} terminology carries a connotation today of a collection---the extension of a concept is nowadays commonly understood in set theory as the set or class of fulfilling instances of that concept. Nevertheless, this collection connotation was not part of Frege's conception. For Frege, the extension object $\varepsilon F$ of a concept $F$ was simply an object taken to represent the concept and needn't itself be understood as a set, class, or collection.

Importantly, Frege insisted that the assignment of extension objects to concepts should be a conceptual invariant---the extensions of equivalent concepts should be the same, and the extensions of inequivalent concepts should be different. In other words, $\varepsilon F$ should represent the concept $F$ extensionally, rather than intentionally, and this is expressed by his Basic Law~V.

\theoremstyle{theorem}
\newtheorem*{BasicLawV}{Basic Law~V}
\begin{BasicLawV}
The extension of a concept $F$ is the same as the extension of concept $G$ if and only if $Fx\iff Gx$ for every $x$.
$$\varepsilon F=\varepsilon G\quad\Iff\quad\forall x\, (Fx\iff Gx).$$
\end{BasicLawV}

Let me refer to the problem of finding a system of extension assignments $F\mapsto\varepsilon F$ fulfilling Basic Law~V as the \emph{extension-assignment problem}. Are we able in general to define extension objects for all our concepts in a way that fulfills this conceptual invariance requirement?


Basic Law~V is an instance of the more general process of \emph{Fregean abstraction}, by which we move from a notion of equivalence to the abstraction of that equivalence, a classification invariant for that equivalence. With Basic Law~V, concepts $F$ and $G$ are logically equivalent in the sense that they have the same course of values $\forall x\ (Fx\iff Gx)$ if and only if they have the same corresponding abstraction, that is, the same extensions of those concepts $\varepsilon F=\varepsilon G$. Similarly, according to the Cantor-Hume principle, considered in section~\ref{Section.Cantor-Hume}, concepts $F$ and $G$ are \emph{equinumerous} in that there is a one-to-one correspondence of the instances of $F$ with instances of $G$ if and only if they have the same corresponding abstractions, that is, the same number of elements $\num F=\num G$. And more generally, for any notion of equivalence $\sim$, the corresponding Fregean abstractions $\alpha F$ would be classification invariants in the sense that $F\sim G$ if and only if $\alpha F=\alpha G$.

In this article, I shall provide deflationary accounts for all of these various Fregean abstraction principles, providing \ZF\ definitions of suitable extension objects $\varepsilon F$, number objects $\num F$, and the various abstraction objects $\alpha F$ for any concept of equivalence, and prove as \ZF\ theorem schemes that they fulfill the desired classification invariant requirements expressed by Basic Law~V, the Cantor-Hume principle, and Fregean abstraction generally. I describe this account as deflationary because the abstractions I provide are expressible and defined in set theory, and as a consequence they are entirely eliminable---every assertion in the language with the abstraction operators is equivalent to an assertion in the plain language of set theory without them.


\section{Basic Law~V holds in \ZF\ set theory}

Let me prove the main result, namely, that Basic Law~V holds in the set-theoretic Fregean framework I have described. Namely, I shall define an explicit association of definable classes $F$ with set objects $\varepsilon F$ representing them in such a way that the association $F\mapsto\varepsilon F$ is definable in set theory in various senses and such that Basic Law~V is fulfilled. The theorem thus realizes Basic Law~V as a provable theorem scheme in Zermelo-Fraenkel set theory.\goodbreak

\newpage
\begin{theorem}\label{Theorem.Basic-law-V-in-ZF}
In Zermelo-Fraenkel \ZF\ set theory, there is an explicit association of definable classes $F$ with set objects $\varepsilon F$, such that:
\begin{enumerate}
  \item The association $F\mapsto\varepsilon F$ is definable in the following senses:
    \begin{enumerate}
    \item For any intentionally given class $F=\set{x\mid\varphi(x,z)}$, the extension object is first-order definable in set theory, using the same parameter $z$ as provided in the description of $F$:
     $$\forall y\ \bigl(\ y=\varepsilon F\quad\text{if and only if}\quad \theta(y,z)\ \bigr).$$
    \item For any class $I$ and class function $i\mapsto F_i$, presented uniformly as a definable class $F=\set{(i,x)\mid i\in I, x\in F_i}\of I\times V$, the function $i\mapsto \varepsilon F_i$ is first-order definable in set theory, again with the same parameters used to define $F$.
    \item The full mapping $F\mapsto\varepsilon F$ is second-order definable in any model of \ZF\ equipped with its definable classes:
    $$\forall F\forall y\ \bigl(\ y=\varepsilon F\quad\text{if and only if}\quad\Theta(F,y)\ \bigr)$$
    \end{enumerate}
  \item The association $F\mapsto\varepsilon F$ fulfills Basic Law~V:
$$\varepsilon F=\varepsilon G\quad\text{ if and only if }\quad \forall x(Fx\iff Gx).$$
\end{enumerate}
\end{theorem}

\begin{proof}
To prove the theorem, I shall begin by explaining how to define the extension objects $\varepsilon F$. Suppose that we are presented with a particular definable class $F=\set{x\mid\varphi(x,z)}$, the class of all sets $x$ fulfilling the first-order formula $\varphi(x,z)$, using a fixed set parameter $z$. To have provided a class concept means to have provided such a formula $\varphi$ and parameter $z$ explicitly. I shall associate this class $F$ with a certain set $\varepsilon F$ in such a manner so as to fulfill all the claims of the theorem. We cannot in general take the class $\set{x\mid \varphi(x,z)}$ itself as the extension, since this might be a proper class, not a set, but we need $\varepsilon F$ to be a set. Instead, what I propose is to use the data of the presentation itself to define $\varepsilon F$.

Let us fix in the metatheory a particular enumeration $\psi_0,\psi_1,\psi_2,\ldots$ of the formulas of the first-order language of set theory. The formula $\varphi$ appears somewhere in this enumeration. Let $\psi_n$ be the earliest formula in that enumeration that is capable, with some choice of parameter, of defining the same class that $\varphi$ defines with parameter $z$. That is, $\psi_n$ is the earliest formula on the list for which there is some set $p$ for which $\forall x\bigl(\varphi(x,z)\iff\psi_n(x,p)\bigr)$. We may find such a parameter $p$ of minimal possible rank in the set-theoretic hierarchy. Indeed, let $u$ be the set of all such parameters $p$ of that minimal rank that define the same class as $\varphi$ does with parameter $z$ using that formula $\psi_n$. We now define $\varepsilon F$ be the pair $(\gcode{\psi_n},u)$, where $\gcode{\psi_n}$ is the set object coding the formula $\psi_n$. In short: the extension of a concept is the earliest formula on the list that is capable of defining it with some parameter, together with the set of minimal-rank parameters that do so.

Let us now prove that $\varepsilon F$ is first-order definable in the senses mentioned in the theorem. Suppose that we are presented with a class $F=\set{x\mid\varphi(x,z)}$, presented intentionally, that is, we have been provided explicitly with the defining formula $\varphi$ and the parameter $z$. The formula $\varphi$ itself appears as some $\psi_n$ in the list, and there are only finitely many preceding formulas $\psi_0,\ldots,\psi_n=\varphi$ in the enumeration. To express the property $y=\varepsilon F$, therefore, we may take a disjunction over $i\leq n$ of the assertions that say,
\begin{quote}
``$y$ is a pair $(\gcode{\psi_i},u)$, whose first coordinate is the code of a formula $\psi_i$ and whose second coordinate $u$ is a nonempty set of parameters $p\in u$ for which $$\forall x\ \bigl(\varphi(x,z)\iff\psi_i(x,p)\bigr),\text{ and}$$
and furthermore $u$ is the set of all such minimal-rank $p$ for which this is true, and finally, none of the formulas $\psi_j$ for $j<i$ admit such a parameter.''
\end{quote}
This property exactly describes what it means for $y$ to the be the extension object $\varepsilon F$ as we defined it above, and since we have the formula $\varphi$ as a metatheoretic formula in hand and there are only finitely many preceding formulas $\psi_i$, all of this is expressible by a formula $\theta$ in the first-order language of set theory
$$y=\varepsilon F\quad\Iff\quad \theta(y,z).$$
Thus we establish claim (1a) as a \ZF\ theorem scheme.

Next, suppose that we are provided with a definable class functional $i\mapsto F_i$, associating every object $i$ in a class $I$ with a class $F_i$. We assume that this class functional has been provided intentionally, by means of a uniform defining formula $\varphi$ and parameter $z$, such that
 $$x\in F_i\quad\Iff\quad \varphi(i,x,z).$$
The uniformity of this definition means that it simply reduces to case (1a), for we can simply treat $i$ as another parameter, thereby defining the extensions
 $$y=\varepsilon F_i\quad\Iff\quad \theta(y,i,z)$$
using the same analysis as above. This provides a uniform definition of the class function
 $$i\mapsto\varepsilon F_i.$$
Thus, whenever we have a concept for a function from objects to concepts, we also have a concept for the extensions of those concepts, establishing (1b) as a \ZF\ theorem scheme.

For statement (1c), let us assume we work in a model of \ZF\ equipped for the Henkin semantics with all and only its definable classes, a model of Gödel-Bernays set theory (except not necessarily with the global choice principle). I claim that the map $F\mapsto\varepsilon F$ is second-order definable in this context. Here we assume that we are given the class $F$ as a class only, without knowing the formula and parameter by which it was defined. Nevertheless, with a second-order assertion ranging over the definable classes, we shall be able to recover the defining formula and thereby recognize and provide the extension object $\varepsilon F$. We may assume that the meta-theoretic enumeration of formulas $\psi_0,\psi_1,\psi_2,\ldots$ that we used in the definition is also available internally to the object theory, arising as the standard part of a definable enumeration. Note that for standard $k$, there is a definable $\Sigma_k$-truth predicate. That is, we have first-order definable classes providing these partial truth predicates. Using these classes, we can express the relation $y=\varepsilon F$ as follows:
\begin{quote}
 $y$ is a pair $\<\gcode{\psi_n},u>$, where $\psi_n$ is a $\Sigma_k$ formula for some $k$, such that there is a class $T$ that is a $\Sigma_k$ truth predicate, $u$ is the set of minimal-rank parameters $p$ such that $\forall x\bigl(Fx\iff T(\gcode{\psi_n},\<x,p>)\bigr)$, and no preceding formula $\psi_i$ has this property.
\end{quote}
The key points are first, that we can express the truth of $\psi_n(x,p)$ using the $\Sigma_k$ truth predicate $T$, if $\psi_n$ is indeed a $\Sigma_k$ assertion, and second, because $F$ is actually definable, there will be a defining formula $\psi_n$ on the list of some standard complexity $\Sigma_k$, and so there will be such a truth predicate class $T$ as needed.
So we have proved (1c).

Finally, let me prove that this definition fulfills Basic law V. We had defined $\varepsilon F$ to be $\<\gcode{\psi},u>$ where $\psi$ is earliest formula to define $F$ and $u$ is set of minimal-rank parameters that do so. Thus, the extension object $\varepsilon F$ depends only on the class $F$ considered as a predicate, and not on the original presentation of $F$ as defined by a formula and parameter $\varphi(x,z)$. If we had used a different formula and parameter, but thereby defined the same class, then the extension object $\varepsilon F$ would still be picking out the least formula $\psi_n$ in the metatheoretic enumeration that is able to define that class, and $u$ would again be the set of minimal-rank parameters that work with this formula to define $F$. Therefore, different equivalent presentations of the definable class will get assigned the same object:
$$\varepsilon F=\varepsilon G\quad\text{ if and only if }\quad \forall x(Fx\iff Gx).$$
And this is precisely what it means for Basic Law~V to be fulfilled.
\end{proof}

Notice how the method exhibits Frege's distinction between identity of concepts and extensional equivalence of concepts. Namely, we can often describe a class in two distinct manners $\set{x\mid\varphi(x)}=\set{x\mid\psi(x)}$, that is, using different formulas $\varphi$ and $\psi$ or perhaps using different parameters, or the same formula with different parameters. So we have given different names to this class, by picking it out by means of different descriptions, and in this sense we have different concepts, but those concepts have the same extension in the case that $\forall x\ \bigl(\varphi(x)\iff\psi(x)\bigr)$, and in this case we shall have $\varepsilon\set{x\mid\varphi(x)}=\varepsilon\set{x\mid\psi(x)}$, or perhaps it is better to write it simply as $\varepsilon\varphi=\varepsilon\psi$.

I described in theorem~\ref{Theorem.Basic-law-V-in-ZF} an interpretation of extension objects for unary classes $\set{x\mid\varphi(x,z)}$ defined by a formula $\varphi(x,z)$ with one free variable $x$ ranging over the defined class and the other variable $z$ held constant as a parameter. In the most general case, however, one would naturally want to consider higher-arity formulas $\varphi(x_1,\ldots,x_k,z_1,\ldots,z_r)$ defining a $k$-ary relation in the first $k$ variables, using the other variables $z_i$ as constants or parameters in the definition. For this one could introduce a formalism $\varepsilon_{x_1,\ldots,x_n}\varphi$ to indicate exactly which variables are taken as part of the defined relation, with the others held as parameters in the definition (and Frege has an analogous notation for exactly this kind of case). Without saying much more about the details of such a system, it is clear that the argument of theorem~\ref{Theorem.Basic-law-V-in-ZF} works essentially just the same with this more general treatment.

What I claim about theorem~\ref{Theorem.Basic-law-V-in-ZF}, furthermore, is that the use of the extension operator $\varepsilon$ in the semantics that are provided is entirely eliminable---every assertion in the language of set theory expanded to allow the extension operator $\varepsilon\varphi$ and more generally the higher-arity versions $\varepsilon_{x_1,\ldots,x_n}\varphi$ is equivalent to an assertion in the base language of set theory alone.

\begin{corollary}\label{Corollary.Deflationary}
 In the semantics for extensions of definable classes $\varepsilon\varphi$ provided by theorem~\ref{Theorem.Basic-law-V-in-ZF}, including the higher-arity versions $\varepsilon_{x_1,\ldots,x_n}\varphi$, every use of the extension operator $\varepsilon$ is eliminable. That is, every assertion in the expanded language of set theory allowing the extension operator $\varepsilon$ is equivalent to an assertion in the language of set theory. Furthermore, one can find the $\varepsilon$-free translation by a primitive recursive translation function.
\end{corollary}

\begin{proof}
One can prove this by induction on formulas, using statements (1a) and (1b) of theorem~\ref{Theorem.Basic-law-V-in-ZF}. One can systematically eliminate $\varepsilon\varphi$ from formulas $\varphi$ not involving $\varepsilon$. The point is statements (1a) and (1b) provide $\varepsilon$-free identity criterion for $y=\varepsilon\varphi$, and so any assertion made about this extension object can be made equivalently without need for any $\varepsilon$ expression. The same idea works in the higher-arity case $\varepsilon_{x_1,\ldots,x_n}\varphi$. The translation process is clearly primitive recursive, being defined by a syntactic recursion on formulas.
\end{proof}

Said differently, the claim is that Basic Law~V, expressed in the language expanding the language of set theory with the extension operator $\varepsilon\varphi$, is a conservative extension of Zermelo-Fraenkel set theory in the base language of set theory.

\section{Prior art}\label{Section.Prior-art}

Let me mention some relevant earlier work in the same direction as the project of this article. Terence Parsons~\cite{Parsons1987:On-the-consistency-of-first-order-Frege} showed the consistency of the first-order fragment of Frege's system by providing a model in which Basic Law~V is fulfilled for first-order definable classes. He begins with the observation that every countable model in a countable language admits, of course, only countably many definable classes. Given a countably infinite model, he divides the domain (externally to the model) into infinitely many infinite classes, reserving these objects to be assigned as extensions $\varepsilon F$ for the definable classes $F$ in a series of stages stratified by the $\varepsilon$ depth of the definitions. That is, at the first stage, using the first class of objects, he assigns extensions for the definable classes that are defined by formulas not involving extensions at all; at the next stage, using the next class of objects, he handles classes defined by formulas that make reference to the extensions, but only the extensions of extension-free formulas; and so on. The main point is that at each stage, all the relevant extension objects for a given formula to be successfully interpreted will have been already assigned at the earlier stages, and in this way he is able to assign all the extension objects $\varepsilon F$ in turn, realizing Basic Law~V.

John L. Bell~\cite{Bell1994:Fregean-extensions-of-first-order-theories} strengthens the Parsons result with a very clear argument, expanding the language Henkin-style with constants $c_A$ for each formula $A$, in such a way that $c_A=c_B\iff\forall x(Ax\iff Bx)$, which is to say that the objects $c_A$ serve as conceptually-invariant extensions for the classes $A$.

John Burgess~\cite{Burgess1998:On-a-consistent-subsystem-of-Freges-Grundgesetze} carries out the Parsons argument a little more explicitly, using compactness to achieve consistency in expanded language. Burgess~\cite[p.~89]{Burgess2005:Fixing-Frege} mounts a similar cardinality argument using \Lowenheim-Skolem to reduce from full second-order to Henkin semantics with countably many classes.

I greatly admire this earlier work, but what I should also like to mention about it is that none of these prior arguments provide a deflationary solution of the extension-assignment problem in the sense I have claimed for my solution in theorem~\ref{Theorem.Basic-law-V-in-ZF}. In the Parsons construction, for example, the labeling of concepts with objects is performed from outside the model, as though by the hand of God (or at least the hand of Parsons), and  similarly in effect for the Bell construction and the compactness arguments. These methods consequently do not provide a uniform definability of the identity criterion, as in (1a), (1b), and (1c) of theorem~\ref{Theorem.Basic-law-V-in-ZF}. In the case of Bell's construction, of course one gets the expressibility of $y=\varepsilon A$ in the expanded language simply by using the formula $y=c_A$, but in general there is no guarantee that the individual named by $c_A$ is definable in the original language (and indeed the model may have no definable elements), and the method does not achieve a uniform identity criterion with respect to parameters. That is, if one defines classes $A_i$ uniformly by a formula $A_i(x)\iff \varphi(x,i)$, then there will generally be no uniform way to express $y=c_{A_i}$ in the manner of (1b) of theorem~\ref{Theorem.Basic-law-V-in-ZF}, since each parameter $i$ leads to a syntactically different constant $c_{A_i}$, and these constants cannot all appear in one formula.

\section{The identity criterion and definability}

I find it frankly strange to neglect the identity criteria and definability questions arising with the extension-assignment problem, for it seems to me that having an expressible identity criterion should be taken as a core part of what it means to solve the extension-assignment problem. Not only must there be an assignment of extension objects $\varepsilon F$ to concepts $F$, but one should also be able to recognize which objects are extensions, which objects go with which concepts, which concepts go with which objects and so on.

In order actually to use the extensions $\varepsilon F$, after all, to form the extensions for our concepts at hand and to use them in our analysis, we would want to know not just that there are objects out there, somewhere, that have been assigned to concepts in a conceptually invariant manner, somehow, with nothing more to be said about which objects are used or how. Rather, we would want to know much more about it. Which objects get assigned as extensions for which concepts? How do we recognize for a given concept $F$ which object is its extension?
      $$y=\varepsilon F$$
How do we recognize for a given object $y$ for which concept $F$ it is the extension?
            $$y=\varepsilon F$$
How do we recognize whether a given object $y$ is the extension of any concept?
          $$\exists F\ y=\varepsilon F$$
What use will an extension assignment $F\mapsto\varepsilon F$ be, after all, if performed externally to the domain of discourse, if it is invisible within it? If the assignments are made Parsons-style from outside the model, as it were, then we cannot seem to work with the resulting solution in the original object theory. But what good is that? The only way to find out which object is the extension $\varepsilon F$ of an expressible concept $F$, in those systems, is to already know what it is.

I view the identity criterion problem for a given solution to the extension-assignment problem as akin to the famous Julius Ceasar problem. Namely, when Frege despairs at the impossibility of determining whether Julius Ceasar is a number, at bottom he is pointing out his lack of an identity criterion, his lack of a decision procedure for determining whether a given individual is a number. Similarly, a solution of the extension-assignment problem without a corresponding identity criterion will lack a decision procedure for determining whether a given object or individual is the extension of a concept.

My view is that the definability of the extension assignments amounts to having what I have referred to as a deflationary account of extensions. When we have an extension-assignment $F\mapsto\varepsilon F$ that admits an expressible identity criterion, one for which we achieve the definability properties mentioned in theorem~\ref{Theorem.Basic-law-V-in-ZF}, then we are in effect able to refer to the extensions of concepts in the original language and theory. Precisely because the features of the extension assignment are expressible in the base language, we have no need to expand the language to include direct reference to extensions. In short, to have a solution of extension-assignment problem that admits of definable identity criterion is exactly to be able to reduce the treatment of extensions to the original language and theory, to deflate the extensions to the original theory, as described in corollary \ref{Corollary.Deflationary}.

For this reason, I find the solution of the extension-assignments provided by my theorem~\ref{Theorem.Basic-law-V-in-ZF} to be an advance on the earlier work mentioned in section~\ref{Section.Prior-art}, at least for the context of set theory as the base theory. Not only can we have a solution of the extension-assignment problem realizing Basic Law~V in set theory, but we can do so in a deflationist manner, so that we can recognize in the object language which objects get assigned to which concepts.

Frege builds the expressivity of the extension assignment into his syntax for the expanded language with the extension operator $\varepsilon$, because he allows himself to form expressions such as $y=\varepsilon\varphi(x,p)$ for any formula $\varphi$, including any formula in the expanded language. This formal assertion expresses the relation of $y$ being the extension of the concept defined by $\varphi(x,p)$. But ultimately one must recognize that Frege provided no semantics for these expressions and did not seem to make any claims about whether his use of $\varepsilon$ should be eliminable.

To my way of thinking, merely having a syntactic expression $\varepsilon\varphi(x,p)$ serving as a name to refer to the extension object of a given concept is quite different from knowing which object it is or whether it is possible to have a semantics for these expressions having all the properties claimed of them. Ultimately, I shall argue, Frege's syntactic approach led him astray, because it was built into his language that he could express the concept ``object $x$ falls under the concept of which $y$ is the extension,'' and it is this, rather than Basic Law~V, that is the true downfall of his system, as I shall explain in theorem~\ref{Theorem.Russell-refutes-BLV}.

\section{Reconciling Russell}

Russell is commonly described as having refuted Frege's system and he is often described specifically as having shown Basic Law~V to be inconsistent. How are we to reconcile this with my claim in theorem~\ref{Theorem.Basic-law-V-in-ZF} that Basic Law~V is a provable theorem scheme of Zermelo-Fraenkel set theory? Wouldn't this show \ZF\ also to be inconsistent?

To be sure, Russell's argument is most naturally construed as a refutation of the general comprehension principle.

\theoremstyle{theorem}
\newtheorem*{GenComp}{General Comprehension Principle}
\begin{GenComp}
 For any property $\varphi$, one may form the set of all $x$ with property $\varphi(x)$. That is,
 $$\set{x\mid\varphi(x)}$$
 is a set.
\end{GenComp}
Russell's argument refutes even very simple instances of this, in the case of the formula $x\notin x$. Namely, by the general comprehension principle, we may form the set $R$ consisting of all sets $x$ that are not elements of themselves. $$R=\set{x\mid x\notin x}.$$
If this is indeed a set, then $R\in R$ if and only if $R\notin R$, which is a contradiction. Thus, the general comprehension principle has contradictory instances.

Although this is often how the Russell paradox is presented, the historical difficulty is that Frege does not state the general comprehension principle explicitly as part of his system. Rather, that principle is essentially hidden away into his manner of forming concepts and denoting them. In his letter replying to Russell, Frege blames Basic Law~V, yet also hints at a systemic problem:
\begin{quote}\label{Frege-quote}
It seems accordingly that the transformation of the generality of an identity into an identity of ranges of values (sect 9 of Gg) is not always permissible, that my law V (sec 20) is false, and that my explanations in sec 31 do not suffice to secure a meaning for my combinations of signs in all cases.

\hfill--Frege, 22 June 1902 letter to Russell,

\hfill responding to Russell's fateful letter
\end{quote}

Let me explain how Russell's argument can be used to refute Basic Law~V in combination with a specific comprehension-like principle.

\begin{theorem}[Russell]\label{Theorem.Russell-refutes-BLV}
Basic Law~V is inconsistent with the existence of the concept ``$x$ does not fall under any concept of which $x$ is the extension.''
\end{theorem}

\begin{proof}
Assume Basic Law~V and let $R$ be the concept ``$x$ does not fall under any concept of which $x$ is the extension.'' If this is indeed a concept, then let $r=\varepsilon R$ be the extension of it, then by Basic Law~V we see that any concept of which $r$ is the extension will agree extensionally with $R$, and so $r$ falls under $R$ if and only if it does not, which is a contradiction.
\end{proof}

The main point here is that Russell's argument is not at its core a refutation of Basic Law~V as such, but rather a refutation of a certain concept-formation principle, an instance of class comprehension. It would suffice for Russell's argument if we had the \emph{falling-under} concept: object $x$ falls under the concepts of which $y$ is the extension. In this case, we could negate and project onto the diagonal to form the forbidden concept mentioned in theorem~\ref{Theorem.Russell-refutes-BLV}.

In fact, the argument does not require the full biconditional of Basic Law V, but only the forward implication, namely, that if $r$ is the extension of a concept, then that concept agrees with $R$. For example, it would suffice that every concept $F$ had a distinct extension $\varepsilon F$. This situation falls far short of Basic Law~V, and it is similar to the situation of every formula $\varphi$ having a distinct Gödel code $\gcode{\varphi}$. The point is that Russell's argument shows that even this weak form of the law is inconsistent with the existence of the concept ``$x$ falls under the concept of which $y$ is the extension.'' In this sense, Russell's argument has little to do with Basic Law~V as such.

I should like to emphasize that the falling-under concept amounts to a proto-truth predicate. That is, to say that object $x$ falls under the concept of which $y$ is the extension is exactly to say that that concept is true of $x$. The extension-assignments $\varepsilon F$ are a kind of Gödel-coding for predicates $F$, and the falling-under concept is the corresponding truth predicate.

The true mistake of Frege's system in my view is not Basic Law V as such, but rather the fact that Frege has built this proto truth predicate into his syntax. That is, the falling-under concept is easily expressed in Frege's system: $x$ falls under the concept of which $y$ is the extension is expressed simply as
  $$\exists F\, (y=\varepsilon F\text{ and }Fx).$$
Furthermore, Frege makes pervasive use of the falling-under concept. He uses falling-under even to define what it means to be a natural number: a natural number is any object that falls under every property that holds of $0$ and is passed from every object to its successor.

Perhaps Frege looked upon the idea of ``falling under'' as innocent, it being merely a special case of functional application, which Frege took as fundamental. That is, if one understands concepts as functionals mapping objects to \emph{true} or $\emph{false}$, then for an object $x$ to fall under a concept $F$ is simply to compute the value of $F(x)$. If this is \emph{true}, then $x$ falls under $F$, and if it is \emph{false}, then it does not.

To be sure, the difficulty of the falling-under relation is not when it is used in a particular case with a particular concept $F$, for one can express that $x$ falls under the concept $F$ simply by $Fx$. Rather, the source of all the trouble is the falling-under relation as a relation of objects to objects, taking the latter objects to represent concepts, as in the relation ``object $x$ falls under the concept of which $y$ is the extension.'' This form of the falling-under relation enables Russell's refutation as described in theorem~\ref{Theorem.Russell-refutes-BLV} and is what amounts to a truth predicate.

\section{Proving Tarski's theorem directly from Russell}

Let me explain next how this kind of Russellian reasoning can be used to provide a direct proof of Tarski's theorem on the nondefinability of truth. Since Gödel's incompleteness theorem can be seen as a consequence of Tarski's theorem, this argument therefore also provides a proof of Gödel's theorem directly from Russell. In short, I would like to explain how truly close Russell was to proving Gödel's theorem.

I find this remarkable, because contemporary accounts of Tarski's theorem usually invert this presentation, commonly proving Tarski's theorem only after Gödel's, and indeed often treating Tarski's theorem essentially as an afterthought to Gödel. A typical proof of Tarski's theorem proceeds by arguing that if there were a truth predicate, then by the Gödel fixed-point lemma there is a sentence asserting its own nontruth according to the predicate, and then by the Liar-paradox reasoning this sentence would be true if and only if the predicate declared it not to be true, contradicting what it means to be a truth predicate.


Here I shall instead prove Tarski's theorem directly by Russell's method, bypassing Gödel, with no explicit use of self-reference or the fixed-point lemma.

\begin{theorem}[Tarski]\label{Theorem.Tarski}
  In Zermelo-Fraenkel set theory, truth is not definable. That is, for any representation of formulas $\varphi$ with objects $\gcode{\varphi}$, there is no formula $T(y,x)$ in the first-order language of set theory such that
   $$\forall x\bigl[T(\gcode{\varphi},x)\iff\varphi(x)\bigr]$$
  for every formula $\varphi(x)$ in the first-order language of set theory.
\end{theorem}

In this set-theoretic context, we have available a variety of natural representations of formulas $\varphi$ with objects $\gcode{\varphi}$ in set theory. Of course there is the Gödelian arithmetization of syntax, taking $\gcode{\varphi}$ as the numerical Gödel code of the formula. But in this set-theoretic context we needn't rely on arithmetization specifically, since there are also softer coding methods available, simply using set theory as it usually is as a foundation of mathematics. That is, we can use the standard set-theoretic interpretations of sequences, functions, and so on, which of course provide set-theoretic interpretations of parse trees and the other syntactical constructions. In this account, $\gcode{\varphi}$ is simply the formula $\varphi$ as it is represented and interpreted in set theory. It is easy to see by metatheoretic induction on formulas that every formula has such a representation in set theory and that basic syntactical operations on formulas are expressible in set theory on those representations. So there are many robust methods to represent formulas $\varphi$ by objects $\gcode{\varphi}$.

I should like to emphasize, however, that the theorem as I have stated it does not require any elaborate development of syntax coding, and it is not about any specific coding of formulas into the object theory. Furthermore, the theorem makes no assumptions about the expressibility of syntactic operations in the codes. Rather, the theorem is a sweeping claim about all possible representations of formulas $\varphi$ with objects $\gcode{\varphi}$, namely, that for no such representation will truth be expressible.

Although I have stated the theorem for Zermelo-Fraenkel set theory \ZF, nevertheless the axiom of infinity is not needed at all, and one can formulate the result in (weak fragments of) finite set theory $\ZFfin$, which is bi-interpretable with Peano arithmetic \PA, and indeed in \PA\ itself, and in weak fragments of \PA. (In any case, the nondefinability of truth passes from any theory to all of its subtheories.) 

\begin{proof}
Suppose toward contradiction that there were such a truth predicate $T$ for some such representation of formulas $\varphi$ by objects $\gcode{\varphi}$. Using this predicate, let $R(x)=\neg T(x,x)$, which is the formula in a sense expressing, ``it is not the case that $x$ is a formula that is true of itself.'' Let $r=\gcode{R}$, and consider the sentence $R(r)$, which asserts $\neg T(r,r)$, which is $\neg T(\gcode{R},r)$, which by the truth-predicate requirement is equivalent to $\neg R(r)$, a contradiction.
\end{proof}

I find this argument to be a direct parallel of Russell's argument. The proof here is essentially identical to Russell's proof of theorem~\ref{Theorem.Russell-refutes-BLV}, simply using $\gcode{R}$ in place of $\varepsilon R$. The Russell formula $R(x)$ asserting that $x$ is not a formula that is true of itself is directly analogous to the Russellian concept of all concepts that do not apply to themselves, or the set of all sets $x$ that are not members of themselves.

A very similar manner of presenting Russell's argument as a proof of Tarski's theorem is also provided in~\cite{Fitting2017:Russells-paradox-Godels-theorem}, and accords with arguments of Smullyan, who had emphasized that much of the fascination surrounding Gödel's theorem may be more properly directed at Tarski's theorem.

Of course, one can find key elements of Gödel's argument here. The construction of the sentence $R(r)$, for example, where $R(x)=\neg T(x,x)$ and $r=\gcode{R(x)}$, is exactly what appears in the usual proof of Gödel's fixed-point lemma, and so the resulting sentence in this Russellian argument is the same as what one finds in the Gödelian fixed-point proof of Tarski's theorem. In that sense, the two approaches to proving Tarski's theorem are not ultimately very different. Nevertheless, the straightforward Russellian treatment above is simpler than the Gödelian proof and entirely avoids the issue of self-reference. In any case, I find it far more natural to view Gödel's fixed-point lemma as generalizing Russell's method here, rather than viewing Russell's argument as an application of Gödel's lemma. Surely Gödel found his proof of the fixed-point lemma by contemplating and generalizing Russell's argument.

I have stated Tarski's theorem in a two-dimensional version, making it in effect about the satisfaction relation rather than sentential truth, but let me mention that we can also formulate a one-dimensional purely sentential version of the theorem, which asserts: for any sufficient coding of syntax into the object theory, there is no formula $T(y)$ for which $T(\gcode{\sigma})\iff\sigma$ for every sentence $\sigma$. Proving this version of Tarski's theorem perhaps brings us closer to Gödelian reasoning, because with this version the substitution of the variable is performed in the object theory rather than in the metatheory as above, and for this we require that syntactic operations such as substitution of terms are expressible in the codes. To prove the sentential formulation, suppose there were such a sentential truth predicate $T$. We can then form the Russell formula $R(x)$ asserting ``$x$ is the code of a formula $\varphi$ that is false at the point $x$.'' To express this version of $R(x)$ from the sentential truth predicate, we form the sentence that would express the formula coded by $x$ at the object $x$ itself, by creating a sentence describing the object $x$ (and this uniform definability is part of our requirement on the coding) and then applying the predicate $T$ to that sentence resulting by substituting that object into the formula coded by $x$. Now we let $r=\gcode{R}$, and consider the sentence $R(r)$, which by the definition of $R$ asserts that the formula coded by $r$ is false at the point $r$, which is to say $\neg T(\gcode{R(r)})$, which would by the truth-predicate property be equivalent to $\neg R(r)$, a contradiction.

Some people point out that Tarski's theorem requires the use of Gödel codes $\gcode{\varphi}$ and the arithmetization of syntax, even for the statement of the theorem, and arithmetization is rightly considered a core contribution of Gödel. I am in complete agreement on the profound nature of Gödel's contribution regarding arithmetization; see my remarks at~\cite[p.~233]{Hamkins2021:Lectures-on-the-philosophy-of-mathematics}. A rebuttal to the specific point here, however, is that Tarski's theorem requires far less coding than Gödel's. As I mentioned earlier, the set-theoretic version of Tarski's theorem has no need for coding specifically into arithmetic, but can use set theory as usual as a foundation of mathematics, a usage widely recognized by Hilbert and others before Gödel. Tarski's theorem also has no need for an arithmetization of deduction and the formal proof system, a key technical burden of Gödel's theorem. Furthermore, for the two-dimensional version of Tarski's theorem, as I discussed earlier, we do not even require effectivity or expressibility of the syntactic operations in the coding (this is why my proof of theorem \ref{Theorem.Tarski} involves considerably less coding than appears in \cite{Fitting2017:Russells-paradox-Godels-theorem}, for example, which aims at the sentential version of Tarski's theorem). Rather, Tarski's theorem in this form seems to require only a rudimentary representation of formulas from the metatheory in the object theory.

To my way of thinking, the idea of having an internal representation of formulas in the object theory can be traced to Frege himself. Namely, there is an analogy between the extension $\varepsilon\varphi$ of a concept $\varphi$ and the Gödel code $\gcode{\varphi}$ of a formula, for in both cases we have an internal object-theory representation of a formula or concept from the metatheory. And as I have explained, this is already enough to prove Tarski's theorem as in theorem \ref{Theorem.Tarski}, without any additional involvement of Basic Law V, that these representations respect logical equivalence.

Let me offer a slogan form of Tarski's theorem in a Fregean-style language:
\begin{quote}
  ``Truth is not a concept.''
\end{quote}
Just as Russell's argument showed in theorem~\ref{Theorem.Russell-refutes-BLV} a sense in which the falling-under relation is not a concept, similarly the Russellian proof of theorem~\ref{Theorem.Tarski} shows that truth is not expressible. The core of the argument in both cases is about our lack of uniformity in expressing these notions. That is, in the context of a given instance, we can easily express that a given object $x$ falls under a given concept $F$ by asserting~$Fx$. What we cannot do is express in a uniform manner that ``object $x$ falls under the concept of which $y$ is the extension.'' Similarly, with Tarski's theorem, we can express that a particular formula $\varphi$ holds at a point $x$ by asserting $\varphi(x)$; what we cannot do is express in a uniform manner that ``$y$ represents a formula that is true at $x$.''

Now that we have proved Tarski's theorem directly via Russell, let me deduce Gödel's incompleteness theorem as a consequence. Namely, since truth is not definable, but provability is definable (and for this one much implement the proof system in the object theory), they cannot be the same thing. QED.

In this way, one realizes the power of Tarski's theorem over Gödel's. Whereas Gödel is showing that truth is not the same as provability, a $\Sigma_1$ notion, Tarski shows that truth is not $\Sigma_n$ expressible for any $n$.


\section{The Cantor-Hume principle}\label{Section.Cantor-Hume}

Frege had used his extension objects $\varepsilon F$ and Basic Law~V in order to prove other abstraction principles, including especially the Cantor-Hume principle, upon which he founded his arithmetic, which in many respects was the focal effort of his project. What I would like to do is explain how the ideas of theorem~\ref{Theorem.Basic-law-V-in-ZF} can be used also to fulfill these other Fregean abstraction principles in Zermelo-Fraenkel set theory.

Let me begin with the central instance, the Cantor-Hume principle, also known simply as Hume's principle, which is the Fregean abstraction principle concerning numbers as a classification invariant of the equinumerosity relation. Specifically, one class $X$ is \emph{equinumerous} with another $Y$, written $X\equinumerous Y$, if there is a one-to-one correspondence between them. The \emph{cardinal-assignment problem}, which we might also call the \emph{number-assignment problem}, is the problem of assigning a cardinal or number object $\num X$ to every class $X$, in such a way that it is an equinumerosity invariant, so that equinumerous classes get assigned the same number and nonequinumerous classes get assigned different numbers. The Cantor-Hume principle is the assertion that indeed there is a solution to this problem.

\newtheorem*{cantorhume}{Cantor-Hume principle}
\begin{cantorhume}
 There is a number-assignment scheme, assigning to every class $X$ a number object $\num X$, for which two classes get assigned the same number if and only if they are equinumerous:
 $$\num X=\num Y\qquad\text{ if and only if }\qquad X\equinumerous Y.$$
\end{cantorhume}

Let me begin by mentioning several easy solutions to the cardinal-assignment problem that are available in Zermelo-Fraenkel set theory for special instances of the principle.

Consider, for example, the set-only version of the cardinal-assignment problem. This is completely solved in \ZFC\ using the axiom of choice by the standard treatment of cardinalities in \ZFC. Namely, in \ZFC\ every set $x$ can be well-ordered and is therefore equinumerous with some smallest ordinal, and so we may simply define the cardinality $|x|$ of the set $x$ to be the smallest ordinal with which it is equinumerous. This not only solves the cardinal-assignment problem, but does so in a way that is simultaneously a solution of the \emph{cardinal-selection problem}, the problem of choosing a unique representative from each equinumerosity class. That is, for any set $x$, the cardinality of $x$ defined as the smallest ordinal equinumerous with $x$ is not only an equinumerousity invariant, but it is one of the sets that is itself equinumerous with $x$---this \ZFC\ definition provides a definable way to pick a canonical member in every equinumerosity class of sets.

This same solution works to an extent in \ZF, without the axiom of choice, at least for the well orderable sets---for every such set $x$ we can define the cardinality $|x|$ to be the smallest ordinal equinumerous with it.

But in fact, we do not the axiom of choice at all to solve the cardinal-assignment problem for sets.

\begin{theorem}
 In Zermelo-Fraenkel set theory \ZF\ without the axiom of choice, there is a definable solution of the cardinal-assignment problem for sets, thereby verifying the Cantor-Hume principle for sets. That is, there is a definable assignment
  $$x\mapsto \num x$$
 assigning to each set $x$ a set $\num x$, in such a way that provides a classification invariant for equinumerosity:
  $$\num x=\num y\quad\text{ if and only if }\quad x\equinumerous y.$$
\end{theorem}

The theorem has become a folklore result in set theory, but it seems that Gödel had it by 1951, see~\cite[footnote to Dfn~6.2, added 1951]{Godel1948:Consistency-of-AC-and-GCH-with-axioms-of-set-theory}.
The main idea has become known as Scott's trick, in light of Dana Scott's discovery of it at around that same time, while an undergraduate student of Tarski's in Berkeley, but it would seem reasonable to me for it to be named the Gödel-Scott trick.\footnote{I corresponded with Dana Scott about his trick, and he explained that Tarski put to the class the problem of finding a set-sized invariant for a definable equivalence relation on the set-theoretic universe. Scott's minimal-rank idea solved it, and Tarski was taken with Scott's solution, dubbing it ``Scott's trick,'' but Scott himself was sheepish about the honor, emphasizing to me his desire to be remembered for other mathematical achievements, several of which he itemized to me.} Scott's trick resolves an important issue in the definition of ultrapowers of the set-theoretic universe, and thus continually arises and rearises in large cardinal set theory, which is preoccupied with such ultrapowers as a central method, and it also arises similarly when defining the Boolean ultrapower of the universe by a forcing notion, thus becoming a key technique in the diverse mutual interpretability results provided by forcing.

The essence of Scott's trick is to handle a proper-class sized equivalence class for an equivalence relation defined on a proper class not by picking representatives of each class, which might not be possible without the global choice principle, but rather by considering the set of rank-minimal elements in each class, which will form a set and which will canonically represent that equivalence class.


\begin{proof}
The key idea is to drop the cardinal-selection aspect of the minimal-ordinal solution in \ZFC, and strive instead merely for an equinumerosity invariant as required by the Cantor-Hume principle. For this, we can for any set $x$ define $\num x$ to be the set of rank-minimal sets in the equinumerosity class of $x$. This is simply to use Scott's trick to select a subset of the equiumerosity class of the set. The set of rank-minimal members of the equinumerosity class determines the equinumerosity class and is invariant with respect to equinumerosity, and so it fulfills the Cantor-Hume principle for sets.
\end{proof}

That seems to polish off the set version of the cardinal-assignment problem. But the Cantor-Hume principle, of course, is not just about assigning cardinals to sets, but requires us also to assign numbers to classes, and so we have not yet achieved a proof of the full Cantor-Hume principle in \ZF.

To find a solution for the class case, let's observe first that if the global choice principle happens to hold, for example, if $V=\HOD$, then we can define a well-ordering of the universe in order type $\Ord$, and from this it follows that all proper classes are equinumerous. In this event, therefore, there will be only one additional equinumerosity class to be handled, the proper class equinumerosity class, and so we can simply assign $\num X$ either to be the smallest ordinal equinumerous with $X$, if $X$ is a set, or some default value $\infty$ that is not an ordinal, if $X$ is a proper class. This will solve the cardinal-assignment problem under global choice, fulfilling the Cantor-Hume principle.

If the global choice principle fails, however, then not all proper classes are equinumerous, even in \ZFC. For example, the whole universe $V$ will not be equinumerous with the class of ordinals, since such a bijection would imply global choice.

So let me now explain a completely general solution to the number-assignment problem, for both sets and classes, one which fulfills the Cantor-Hume principle in \ZF\ without using any choice principle at all. The main difficulty here is that we may have many different proper classes, not equinumerous, but we need to associate to each of them a cardinal number in such a way that fulfills the Cantor-Hume principle.

\begin{theorem}\label{Theorem.CantorHume}
In Zermelo-Fraenkel set theory \ZF\ with definable classes only, there is a number assignment scheme $F\mapsto\#F$, assigning to every definable class $F$ a set object $\#F$, called the number of the class $F$, with the following properties:
\begin{enumerate}
  \item For any particular definable class $F=\set{x\mid\varphi(x,z)}$, the identity criterion for being the number $\#F$ is expressible by a formula $\Theta$ in the second-order language of set theory, using the same set parameter $z$, if any, appearing in the presentation of $F$:
   $$y=\#F\quad\text{ if and only if }\quad \Theta(y,z).$$
  \item The identity criterion of whether a given object $x$ is a number is expressible by a formula $\Lambda(x)$ in the second-order language of set theory.
  \item The numbers $\#F$ are well-defined with respect to the equinumerosity equivalence of the classes, in the sense that
       $$\#F=\#G\quad\text{ if and only if }\quad F\text{ is equinumerous with }G.$$
  In other words, the Cantor-Hume principle holds for these number assignments.
\end{enumerate}
\end{theorem}

\begin{proof}
The main difficulty, of course, is to describe how we are to assign numbers to each class. We shall provide a general procedure that works for any definable class $F$. Namely, suppose we are faced with a particular definable class $F=\set{x\mid\varphi(x,z)}$, named intentionally by the formula $\varphi$ with parameter $z$. Let us define $\#F$ to be the pair $\<\gcode{\psi},u>$, where $\psi$ is the first formula on the enumeration of all formulas for which there is some $z$ such that $\set{x\mid\psi(x,z)}$ is equinumerous with $F$, and $u$ is the set of all minimal-rank such sets $z$ for which this is true.

Let's verify statement (1) of the theorem. Given any particular definable class $F=\set{x\mid\varphi(x,z)}$, let $\Theta(y,z)$ assert that $y$ is a pair $\<\gcode{\psi},u>$ whose first coordinate is the code of a formula $\psi$ and whose second coordinate is the set of minimal-rank parameters $p$ for which $\set{x\mid \psi(x,p)}$ is a class that is equinumerous with $F$ by a definable bijection, and such that no earlier formula $\psi_i$ preceding $\psi$ on the list can define with some parameter such a class equinumerous with $F$. This assertion is expressible in second-order set theory, in the context where the classes are all the first-order definable classes, because we can quantify over the partial truth predicates and thereby refer to the existence of a definable bijection between the classes.

For statement (2), the question of whether a given object $x$ is the number $\num F$ of some definable class $F$ is simply the question of whether $x$ is a pair $\<\gcode{\psi},u>$ whose first coordinate is the code of a formula for which there is a truth predicate available and whose second coordinate is the set of minimal-rank parameters all defining via $\psi$ classes that are definably equinumerous with one another, such that no earlier formula works in this way.

For statement (3), the number assignment $\num F$ clearly determines the equinumerosity class of $F$ and depends only on that class and not on the particular presentation of $F$. So it will fulfill the full Cantor-Hume principle as desired.
\end{proof}

One can understand statement (2) as an answer of sorts to the Julius Caesar problem, for it provides a criterion deciding whether a given object $x$ is a number on this account of number. Of course, Frege was concerned with the Julius Caesar problem in so far as that the Cantor-Hume principle, by itself, did not seem automatically to provide an identity criterion for what is a number. Here, in contrast, statement (2) is providing such an identity criterion for the specific number-assignment defined in the proof of this theorem; there is no claim, of course, that this is the only solution to the number-assignment problem and other assignments will naturally have a different identity criterion.

Perhaps one objects to this theorem in comparison with the realization of Basic Law~V in theorem~\ref{Theorem.Basic-law-V-in-ZF} because here we have only second-order definitions for the number identity criterion. But to my way of thinking, this objection must retreat in the face of the fact that the equinumerousity relation itself is a second-order relation. That is, the invariance property of Basic Law~V is the first-order property of logical equivalence of two classes, $\forall x\,[Fx\iff Gx]$, whereas with the Cantor-Hume principle the equivalence relation of equinumerosity $F\equinumerous G$ is itself asking for a definable bijection between $F$ and $G$, a second-order quantification. So it is hardly surprising that the identity criterion is also second-order expressible.

\section{Fregean abstraction in \ZF}\label{Section.Fregean-abstraction}

Frege emphasized the process by which we abstract from a notion of equivalence to the abstraction of that equivalence. We have a notion of parallel lines, for example, which all have the same direction, but what is the \emph{direction} of a line? According to Frege, it is an abstraction of parallelism. The direction of a line is an invariant of the parallelism relation---parallel lines have the same direction and nonparallel lines have different directions. Similarly, we have the notion of concepts being logically equivalent, and the \emph{extension} of a concept is the corresponding abstraction---according to Basic Law~V, logically equivalent concepts have the same extension and logically inequivalent concepts have different extensions. We have the notion of classes being equinumerous, and the \emph{number} of a class is the abstraction---according to the Cantor-Hume principle, equinumerous classes have the same number of elements and nonequinumerous classes have a different number of elements.

Thus Fregean abstraction generalizes Basic Law~V and the Cantor-Hume principle. In the general case, we have an equivalence relation $x\sim y$, and we seek a system of abstractions $x\mapsto\alpha x$ with the property that equivalent instances get assigned the same abstraction and inequivalent instances get assigned different abstractions.

\newtheorem*{FregeanAbs}{Fregean abstraction principle}
\begin{FregeanAbs}
 For any given equivalence relation $x\sim y$, the Fregean abstraction for this relation is a system of abstractions
  $$x\mapsto \alpha x$$
 which respect the equivalence relation:
  $$\alpha x=\alpha y\quad\text{ if and only if }\quad x\sim y.$$
\end{FregeanAbs}

In short, Fregean abstraction for an equivalence relation $\sim$ provides a classification invariant for that relation, an assignment $x\mapsto \alpha x$ that is constant on each equivalence class and with different values on different equivalence classes.

In this section, I should like to investigate how well we might achieve the Fregean abstraction principle in Zermelo-Fraenkel set theory. Let me start as I had earlier with easier special instances of the principle. Let us begin with the easy case of a first-order definable equivalence relation $x\sim y$ on the sets.

\begin{theorem}\label{Theorem.Fregean-abstraction-sets}
 In Zermelo-Fraenkel set theory \ZF, suppose that $x\sim y$ is a first-order definable equivalence relation defined on a class of sets. Then there is definable system of abstractions, $x\mapsto \alpha x$, associating to each set $x$ in the domain an object $\alpha x$ called the abstraction of $x$, which form a classification invariant for $\sim$ and thereby fulfill the Fregean abstraction principle:
  $$\alpha x=\alpha y\quad\text{ if and only if }\quad x\sim y$$
\end{theorem}

\begin{proof}
In this easy case, the equivalence relation lives on sets only, and so we can simply employ Scott's trick. Namely, we define the abstraction $\alpha x$ of any object $x$ in the relevant class to be the set of rank-minimal members of the $\sim$-equivalence class of $x$. This is an invariant for the equivalence relation, because if $x\sim y$, then they have the same equivalence class and consequently the same set of rank-minimal members of that class, and so $\alpha x=\alpha y$. And conversely, the set of rank-minimal members of an equivalence class determines that class, and so if $\alpha x=\alpha y$, then indeed $x\sim y$. So this is a full solution of Fregean abstraction in \ZF\ for definable class equivalence relations on the sets.
\end{proof}

This theorem covers the case of parallel lines, if one regards the lines as interpreted in set theory, and many further cases of that nature, including the case of equinumerosity for sets. In particular, this simple theorem covers the case of equinumerosity for finite classes, since these must be sets, and as a consequence this theorem provides a sufficient level of Fregean abstraction to undertake much of Frege's development of natural number arithmetic, which requires only sets rather than classes, since every finite class is a set.

But of course, several of Frege's other principal applications of Fregean abstraction arise not from equivalence relations on sets, but on classes. This is a more difficult realm, in which Scott's trick is not necessarily applicable. So let me make the step up to classes, treating first the case of a first-order expressible equivalence relation on definable classes: $$F\approx G.$$
The instances $F$ and $G$ here will be definable classes in \ZF\ set theory, and the notion of equivalence $F\approx G$ will be expressible by a formula $\varphi(F,G)$ involving only first-order quantification, in which those classes $F$ and $G$ appear as predicates. This is the kind of equivalence arising in Basic Law~V, for example, because to express the course-of-values equivalence for concepts $F$ and $G$, which is to say, that they are logically equivalent, is to say $\forall x\,(Fx\iff Gx)$, which is therefore first-order expressible in terms of $F$ and $G$.

\begin{theorem}\label{Theorem.Fregean-abstraction-classes-first-order}
 In Zermelo-Fraenkel set theory, suppose that $\approx$ is a first-order definable equivalence relation on definable classes. Then there is a system of abstractions $F\mapsto\alpha F$, associating to each definable class $F$ an abstraction object $\alpha F$, which fulfill the Fregean abstraction principle:
 $$\alpha F=\alpha G\quad\text{ if and only if }\quad F\approx G.$$
 Furthermore, for any particular definable class $F=\set{x\mid\varphi(x,z)}$, the identity criterion $y=\alpha F$ is first-order expressible in set theory, using only the same parameter $z$, if any, appearing in the presentation of $F$:
 $$ \forall y\ \bigl( y=\alpha F\quad\iff\quad \theta(y,z)\bigr).$$
 If $I$ is any class and $i\mapsto F_i$ is a definable sequence of classes, then the identity criterion $y=\alpha F_i$ is first-order expressible in set theory, using the same parameter used to define the sequence:
 $$\forall y\forall i\ \bigl[ y=\alpha F_i\quad\iff\quad\theta(y,i,z)\bigr]$$
\end{theorem}

\begin{proof}
We use the same idea as in theorem~\ref{Theorem.Basic-law-V-in-ZF}. Namely, given any class $F=\set{x\mid\varphi(x,z)}$, define the abstraction $\alpha F$ to be the pair $\<\gcode{\psi_n},u>$, whose first coordinate is the code of a formula $\psi_n$ that with some parameter $p$ defines a class that is equivalent to $F$, that is, for which $\set{x\mid\psi_n(x,p)}\approx\set{x\mid\varphi(x,z)}$, and where $\psi_n$ is the first formula to achieve this with some parameter and $u$ is the set of rank-minimal parameters that do so with $\psi_n$. We don't insist that all parameters in $u$ define the same class, but rather just that they define an $\approx$-equivalent class. This assignment will fulfill the Fregean abstraction principle, because $\alpha F$ is determined by and determines the $\approx$-class of $F$, and so we will achieve $\alpha F=\alpha G\Iff F\approx G$.

The identity criterion $y=\alpha F$ will be first-order expressible, if we are provided with the explicit definition of $F=\set{x\mid\varphi(x,z)}$, since there will be only finitely many earlier competitor formulas in the metatheoretic list of formulas, and to express that a given definable class is $\approx$-equivalent to $F$ is first-order expressible.

That case immediately implies that we can also express the identity criterion $y=\alpha F_i$ for a uniformly definable sequence of classes $i\mapsto F_i$, defined by a first-order formula $F_i(x)\iff\varphi(x,i,z)$. The point is that this amounts simply to regarding $i$ as another parameter.
\end{proof}

As I mentioned, this theorem has Basic Law~V as a special case, since the equivalence relation of logical equivalence for definable classes is first-order expressible. In this sense, theorem~\ref{Theorem.Basic-law-V-in-ZF} is a special case of theorem~\ref{Theorem.Fregean-abstraction-classes-first-order}.

Let me finally treat the most difficult case, where we have an equivalence $F\equiv G$ on definable classes that is second-order expressible in Zermelo-Fraenkel set theory equipped with definable classes only.

\begin{theorem}\label{Theorem.Fregean-abstraction-classes-second-order}
 In Zermelo-Fraenkel set theory, suppose that $\equiv$ is a second-order definable equivalence relation on definable classes, in the context where all classes are first-order definable. Then there is a system of abstractions $F\mapsto\alpha F$, associating to each definable class $F$ an abstraction object $\alpha F$, which fulfill the Fregean abstraction principle:
 $$\alpha F=\alpha G\quad\text{ if and only if }\quad F\equiv G.$$
 Furthermore, the identity criterion $y=\alpha F$ is second-order expressible
 $$y=\alpha F\quad\text{ if and only if }\quad \Theta(y,F).$$
\end{theorem}

\begin{proof}
We may proceed just as in theorem~\ref{Theorem.CantorHume}. Namely, the abstraction object $\alpha F$ is a pair $\<\gcode{\psi_n},u>$ where $\psi_n$ is the first formula on the metatheoretic list that is able with some parameter $p$ to define a class $\set{x\mid\psi_n(x,p)}$ that is $\equiv$-equivalent to $F$. This is easily seen to fulfill the Fregean abstraction principle requirement, since $\alpha F$ determines and is determined by the $\equiv$-equivalence class of $F$. The identity criterion $y=\alpha F$ will be expressible in second-order set theory in our context, because first of all, we have available the partial truth predicates sufficient to determine the truth of any particular formula $\psi_n(x,p)$, and the $\equiv$-equivalence relation was assumed to be second-order expressible.
\end{proof}

One must disparage the fact that the identity criterion $y=\alpha F$ is merely second-order expressible here, which is surely unsatisfying, but it should not be surprising at all in light of the fact that this is the case where the equivalence relation $\equiv$ itself is merely second-order definable. Theorem~\ref{Theorem.Fregean-abstraction-classes-second-order} includes the case of equinumerosity of classes, and thus has theorem~\ref{Theorem.CantorHume} as a special case.

%
%
%
%
%

\section{Gödel-Bernays and Kelley-Morse set theories}

Let us now investigate the extent to which we may achieve the Fregean abstraction principles in the various standard second-order set theories, such as Gödel-Bernays set theory and Kelley-Morse set theory. In the earlier main theorems, we had considered Zermelo-Fraenkel set theory in the context of all its definable classes, which is a model of Gödel-Bernays set theory (without the axiom of choice or global choice). All those arguments generalize, however, to the context of a \emph{principal} model of Gödel-Bernays set theory, which is a model in which there is a class $Z$ from which all other classes are first-order definable. An important example of these kinds of models arising when performing class forcing over a model of \ZFC, for in such a case one gets a principal model of Gödel-Bernays set theory when taking as classes everything that is definable from the generic filter class $G$. Indeed, the principal models are preserved generally by tame class forcing, since one can amalgamate the new generic filter with the old class parameter.

\begin{observation}
 All the results of theorems~\ref{Theorem.Basic-law-V-in-ZF},~\ref{Theorem.CantorHume}, and~\ref{Theorem.Fregean-abstraction-classes-second-order} hold not just in all models of Zermelo-Fraenkel with definable classes, but also in all principal models of Gödel-Bernays set theory, provided one allows a fixed class parameter in the definitions.
\end{observation}

The earlier proofs simply relativize, accommodating the fixed class parameter $Z$ into the definitions of the abstraction objects. The extension $\alpha F$ of a class $F=\set{x\mid \varphi(x,z,Z)}$, for example, will be the pair $\<\gcode{\psi_n},u>$ where $\psi_n$ is the first formula on the list able to define $F$ with some parameter $p$ and class $Z$, and $u$ is the set of rank-minimal such parameters $p$ that do so. And similarly with the other Fregean abstraction invariants $\num F$ and $\alpha F$ for the other instances of Fregean abstraction.

The main result of~\cite{HamkinsLinetskyReitz2013:PointwiseDefinableModelsOfSetTheory} shows that every countable model of Gödel-Bernays set theory (with countably many sets and classes) has an extension to a pointwise definable model, a model of \GBC\ in which every set and class is first-order definable without parameters. Such a model falls under the scope of the original theorem, since the first-order \ZFC\ content of the model suffices to recover definably all the original classes. In this sense, any countable set-theoretic universe can be extended to one where there are definable solutions for Fregean abstraction.

What I should like to observe next, however, is that one cannot achieve this in a model of Kelley-Morse set theory, for Russell's argument can be implemented in \KM\ with force.

\begin{theorem}
 Kelley-Morse set theory is inconsistent with a definable solution of the extension-assignment problem (even second-order definable), which is to say, \KM\ is inconsistent with a definable solution to Basic Law~V.
\end{theorem}

\begin{proof}
Assume in Kelley-Morse set theory that we have a definable assignment $F\mapsto\varepsilon F$, in the sense that
 $$\forall x\forall F\ \bigl[ x=\varepsilon F\quad\iff\quad \theta(x,F)\bigr]$$
for some formula $\theta(x,F)$ in the second-order language of set theory. It would be fine for the argument if this definition involved second-order quantifiers. It follows that the Julius-Caesar-like property of $x$ being an extension $\exists F\, x=\varepsilon F$, is second-order expressible in Kelley-Morse set theory.

Let us also assume that this system of extension-assignment fulfills Basic Law~V:
 $$\forall F,G\ \bigl[ \varepsilon F=\varepsilon G\quad\iff\quad \forall x\, (Fx\iff Gx)\bigr].$$
In the Russellian style, let $R$ be the class of all $x$ that does not fall under any concept of which it is the extension. That is,
 $$R(x)\quad\Iff\quad \forall F\ \bigl[x=\varepsilon F\implies \neg Fx\bigr].$$
If $r=\varepsilon R$, then $R(r)$ asserts that $\neg F(r)$ for any $F$ for which $r=\varepsilon F$, and by the Basic Law~V requirement all such $F$ agree on whether $F(r)$ or not. But $R$ is such an $F$, and so $R(r)$ if and only if $\neg R(r)$, which is a contradiction.
\end{proof}

Theorem~\ref{Theorem.Basic-law-V-in-ZF} shows that the corresponding theorem is not provable in Gödel-Bernays set theory, because that theorem shows that every model of Zermelo-Fraenkel \ZF\ set theory, when equipped with its definable classes, does admit a definable system of extension assignments $F\mapsto\varepsilon F$.

%
%
%
%

\section{Fregean ontology without infinity}

Although I have provided the analysis of this paper for Zermelo-Fraenkel \ZF\ set theory, which includes the axiom of infinity, nevertheless the entire analysis can be undertaken in finite set theory $\ZFfin$, which replaces the infinity axiom with its negation and which includes the foundation axiom in the form of the $\in$-induction schema. I find finite set theory to be quite a nice realm for Fregean thinking.

\begin{observation}
 All the results of theorems~\ref{Theorem.Basic-law-V-in-ZF},~\ref{Theorem.CantorHume}, and~\ref{Theorem.Fregean-abstraction-classes-second-order} hold in models of finite set theory $\ZFfin$, equipped with their definable classes.
\end{observation}

The proofs go through essentially unchanged. One thing that is a bit nicer in finite set theory is that there is a definable global well-ordering---this theory proves that the whole universe is in bijective correspondence with the finite von Neumann ordinals (just as $V_\omega$ can define a bijection of itself with $\omega$), and so we are in the easy case for definably picking representatives in every definable class.

\printbibliography

\end{document}

%% file: MathMacrosJDH.tex
\newtheorem{theorem}{Theorem}
\newtheorem*{theorem*}{Theorem}

\newtheorem*{maintheorem*}{Main Theorem}
\newtheorem*{maintheorems*}{Main Theorems}
\newtheorem{corollary}[theorem]{Corollary}
\newtheorem*{corollary*}{Corollary}
\newtheorem*{corollaries*}{Corollaries}

\newtheorem{observation}[theorem]{Observation}

\theoremstyle{definition}

\newtheorem*{definition*}{Definition}

\newtheorem*{question*}{Question}

\newtheorem*{questions*}{Questions}
\newtheorem*{mainquestion*}{Main Question} 
\newtheorem*{openquestion*}{Open Question} 
\theoremstyle{remark}

\newcommand{\QED}{\end{proof}}

\def\proclaim[#1]{{\bf #1}}
\def\BF#1.{{\bf #1.}}

\def\says#1:#2\par{\item[#1] #2\par}

\newcommand{\Lowenheim}{L\"owenheim}

\makeatletter
\newcommand{\dotminus}{\mathbin{\text{\@dotminus}}}
\newcommand{\@dotminus}{%
  \ooalign{\hidewidth\raise1ex\hbox{.}\hidewidth\cr$\m@th-$\cr}%
}
\makeatother

\newcommand{\of}{\subseteq}

\newcommand{\set}[1]{\{\,{#1}\,\}}

\newcommand\equinumerous{\simeq}

\renewcommand{\setminus}{\raise.3ex\hbox{\rotatebox{-20}{$-$}}} 

\newcommand{\smalllt}{\mathrel{\mathchoice{\raise2pt\hbox{$\scriptstyle<$}}{\raise1pt\hbox{$\scriptstyle<$}}{\raise0pt\hbox{$\scriptscriptstyle<$}}{\scriptscriptstyle<}}}
\newcommand{\smallleq}{\mathrel{\mathchoice{\raise2pt\hbox{$\scriptstyle\leq$}}{\raise1pt\hbox{$\scriptstyle\leq$}}{\raise1pt\hbox{$\scriptscriptstyle\leq$}}{\scriptscriptstyle\leq}}}

   \def\DHLhksqrt#1#2{%
   \setbox0=\hbox{$#1\sqrt{#2\,}$}\dimen0=\ht0
   \advance\dimen0-0.2\ht0
   \setbox2=\hbox{\vrule height\ht0 depth -\dimen0}%
   {\box0\lower0.4pt\box2}}

\def\[#1]{\mathopen{\lbrack\!\lbrack}#1\mathclose{\rbrack\!\rbrack}}
\newbox\gnBoxA
\newbox\gnBoxB
\newdimen\gnCornerHgt
\setbox\gnBoxA=\hbox{\tiny$\ulcorner$}
\global\gnCornerHgt=\ht\gnBoxA
\newdimen\gnArgHgt
\def\gcode #1{%
\setbox\gnBoxA=\hbox{$#1$}%
\setbox\gnBoxB=\hbox{$\bar #1$}%
\gnArgHgt=\ht\gnBoxB%
\ifnum     \gnArgHgt<\gnCornerHgt \gnArgHgt=0pt%
\else \advance \gnArgHgt by -\gnCornerHgt%
\fi \raise\gnArgHgt\hbox{\tiny$\ulcorner$} \box\gnBoxA %
\raise\gnArgHgt\hbox{\tiny$\urcorner$}}
\newcommand{\UnderTilde}[1]{{\setbox1=\hbox{$#1$}\baselineskip=0pt\vtop{\hbox{$#1$}\hbox to\wd1{\hfil$\sim$\hfil}}}{}}
\newcommand{\Undertilde}[1]{{\setbox1=\hbox{$#1$}\baselineskip=0pt\vtop{\hbox{$#1$}\hbox to\wd1{\hfil$\scriptstyle\sim$\hfil}}}{}}
\newcommand{\undertilde}[1]{{\setbox1=\hbox{$#1$}\baselineskip=0pt\vtop{\hbox{$#1$}\hbox to\wd1{\hfil$\scriptscriptstyle\sim$\hfil}}}{}}
\newcommand{\UnderdTilde}[1]{{\setbox1=\hbox{$#1$}\baselineskip=0pt\vtop{\hbox{$#1$}\hbox to\wd1{\hfil$\approx$\hfil}}}{}}
\newcommand{\Underdtilde}[1]{{\setbox1=\hbox{$#1$}\baselineskip=0pt\vtop{\hbox{$#1$}\hbox to\wd1{\hfil\scriptsize$\approx$\hfil}}}{}}

\renewcommand{\implies}{\mathrel{\rightarrow}}

\renewcommand{\iff}{\mathrel{\leftrightarrow}}
\newcommand{\Iff}{\mathrel{\Longleftrightarrow}}

\def\<#1>{\left\langle#1\right\rangle}

\newcommand{\Ord}{\mathord{{\rm Ord}}}

\newcommand{\ZFC}{{\rm ZFC}}
\newcommand{\ZF}{{\rm ZF}}

\newcommand\ZFfin{\ZF^{\neg\infty}}

\newcommand{\KM}{{\rm KM}}

\newcommand{\GBC}{{\rm GBC}}

\newcommand{\HOD}{{\rm HOD}}

\newcommand{\PA}{{\rm PA}}

\newcommand{\cell}[1]{\boxit{\hbox to 17pt{\strut\hfil$#1$\hfil}}}
\newcommand{\head}[2]{\lower2pt\vbox{\hbox{\strut\footnotesize\it\hskip3pt#2}\boxit{\cell#1}}}
\newcommand{\boxit}[1]{\setbox4=\hbox{\kern2pt#1\kern2pt}\hbox{\vrule\vbox{\hrule\kern2pt\box4\kern2pt\hrule}\vrule}}
\newcommand{\Col}[3]{\hbox{\vbox{\baselineskip=0pt\parskip=0pt\cell#1\cell#2\cell#3}}}
\newcommand{\tapenames}{\raise 5pt\vbox to .7in{\hbox to .8in{\it\hfill input: \strut}\vfill\hbox to
.8in{\it\hfill scratch: \strut}\vfill\hbox to .8in{\it\hfill output: \strut}}}
\newcommand{\Head}[4]{\lower2pt\vbox{\hbox to25pt{\strut\footnotesize\it\hfill#4\hfill}\boxit{\Col#1#2#3}}}
\newcommand{\Dots}{\raise 5pt\vbox to .7in{\hbox{\ $\cdots$\strut}\vfill\hbox{\ $\cdots$\strut}\vfill\hbox{\
$\cdots$\strut}}}